\documentclass{article}
\usepackage{amsmath}
\usepackage{amsfonts}
\usepackage{amsthm}
\usepackage{amssymb}
\usepackage[english]{babel}
\usepackage[ansinew]{inputenc}
\usepackage[all]{xy}
\usepackage{color}

\theoremstyle{definition}

\newtheorem{defi}{Definition}[section]
\newtheorem{defin}[defi]{Definition}
\newtheorem{rem}[defi]{Remark}
\newtheorem{nota}[defi]{Notation}

\theoremstyle{plain}

\newtheorem{teo}[defi]{Theorem}
\newtheorem{cor}[defi]{Corollary}
\newtheorem{lem}[defi]{Lemma}
\newtheorem{prop}[defi]{Proposition}

\author{Marco Antei}

\title{Extention of Finite Solvable Torsors over a Curve}
\begin{document}
\maketitle

\noindent \textbf{Abstract.} Let $R$ be a discrete valuation ring with fraction field $K$ and with algebraically closed residue field of positive characteristic $p$. Let $X$ be a smooth fibered surface over $R$ with geometrically connected fibers endowed with a section $x\in X(R)$. Let $G$ be a finite solvable $K$-group scheme and assume that either $|G|=p^n$ or $G$  has a normal series of length $2$. We prove that every quotient pointed $G$-torsor over the generic fiber $X_{\eta}$ of $X$  can be extended to a torsor over $X$ after eventually extending scalars and after eventually blowing up $X$ at a closed subscheme of its special fiber $X_s$.

\tableofcontents
\bigskip

\section{Introduction}\label{sez:Intro} 
Let $S$ be a connected Dedekind scheme and $\eta=Spec(K)$ its generic point; let $X$ be a scheme, $f:X\to S$ a faithfully flat morphism of finite type and $f_{\eta}:X_{\eta}\to \eta$ its generic fiber. Assume we are given a finite $K$-group scheme $G$ and a $G$-torsor $Y\to X_{\eta}$. The problem of extending a torsor $Y\to X_{\eta}$ consists of searching a finite and flat  $S$-group scheme $G'$ whose generic fiber is isomorphic to $G$ and  a $G'$-torsor $T\to X$ whose generic fiber is isomorphic to $Y\to X_{\eta}$ as a $G$-torsor. Some solutions to this problems are known in some particular relevant cases, that we briefly recall hereafter. The first important answer to this problem is due to Grothendieck: he proves that, after eventually extending scalars, the problem has a solution when $S$ is the spectrum of a complete discrete valuation ring with algebraically closed residue field of positive characteristic $p$, with $X$ proper and smooth over $S$ with geometrically connected fibers  and $(|G|,p)=1$ (\cite{SGA1}, Expos\'e X, or \cite{Szamuely}, Theorem 5.7.10).  When $S$ is the spectrum of a discrete valuation ring of residue characteristic $p$, $X$ is a proper and smooth curve over $S$ then Raynaud suggests a solution, after eventually extending scalars, for $G$ commutative of order a power of $p$  (\cite{Ray3} \S 3). A similar problem has been studied by Sa\"idi in \cite{Saidi}, \S 2.4 for formal curves of finite type and $G=(\mathbb{Z}/p\mathbb{Z})_K$.  When $S$ is the spectrum of a  d.v.r. $R$ of mixed characteristic $(0,p)$ Tossici provides a solution, after eventually extending scalars, for $G$ commutative when $X$ is a regular scheme, faithfully flat over $S$, with integral fibers provided that the normalization of $X$ in $Y$ has reduced special fiber (\cite{Tos}, Corollary 4.2.8). Finally in \cite{Antei2}, \S 3.3 we provide a solution for $G$ commutative, when  $S$ is a connected Dedekind scheme and $f:X\to S$ is a relative smooth curve with geometrically integral fibers endowed with a section $x:S\to X$ provided that $Y$ is pointed over $x_{\eta}$ (or, in higher dimension, a smooth morphism satisfying additional assumptions, cf.  \cite{Antei2}, \S 3.2). We stress that in this last case we do not need to extend scalars.\\

In this paper we study the problem of extending the $G$-torsor $Y\to X_{\eta}$ when $G$ is finite and solvable. More precisely the aim of this paper is to prove the following:
\begin{teo}(Theorem \ref{teoPRINCIPALE} and Corollary \ref{corLunghezzaDue})
Let $R$ be a discrete valuation ring with fraction field $K$ and with algebraically closed residue field of positive characteristic $p$. Let $X$ be a smooth fibered surface over $R$ with geometrically connected fibers endowed with a section $x\in X(R)$. Let $G$ be a finite and solvable $K$-group scheme. We prove that every quotient pointed $G$-torsor over the generic fiber $X_{\eta}$ of $X$  can be extended to a torsor over $X$ after eventually extending scalars and after eventually blowing up $X$ at a closed subscheme of its special fiber $X_s$ in the following two cases:
\begin{enumerate}
	\item $|G|=p^n$;
	\item $G$  has a normal series of length $2$.
\end{enumerate}
\end{teo}

\indent \textbf{Acknowledgements:} I would like to thank Mohammed Sa\"idi for inviting me at the Isaac Newton Institute where we had first interesting discussions on this topic. I would like to thank Vikram Mehta for inviting me at T.I.F.R. (Mumbai) where I have developed further this paper. I would like to thank Hausdorff Center (Bonn) for hospitality. Finally I would like to thank Michel Emsalem and Lorenzo Ramero for helpful remarks.

\section{Towers of torsors and solvable torsors}
\begin{nota} Throughout the whole paper every scheme will be supposed locally noetherian. 
Let $X$ be a $S$-scheme, $G$ a flat $S$-group scheme and $Y$ a $S$-scheme endowed with a right action of $G$.  A $S$-morphism $p:Y\to X$ is said to be a $G$-torsor  if it is affine, faithfully flat, $G$-invariant and locally trivial for the fpqc topology. We  say that a $G$-torsor is finite  if $G$ is finite and flat. Likewise we  say that a $G$-torsor is commutative (resp. solvable) if $G$ is commutative (resp. solvable). When we fix a section $x\in X(S)$ we say that a $G$-torsor $p: Y \to X$ is pointed if there exists a section $y\in Y_x(S)$. \end{nota}
\subsection{Solvable torsors}
Let $S$ be any connected scheme, recall that a finite and flat $S$-group scheme $G$ is said to be solvable if it has a normal series (or solvable series) 
\begin{equation}\label{eqSolvab} 0=H_n\triangleleft H_{n-1}\triangleleft .. \triangleleft H_{1}   \triangleleft H_0=G\end{equation} where each $H_i$ is a finite  and flat $S$-group scheme and each  quotient $H_i/H_{i+1}$ exists as an $S$-group scheme and is finite,  flat and commutative ($i=0, .., n-1$).  As usual $n$ is called the length of such a normal series.

\begin{rem}
Recall that if the $H_i$ are finite and flat then each $H_i/H_{i+1}$ exists and is a finite and flat $S$-group scheme (\cite{Sha}, \S 3, Theorem). 
\end{rem}

 Then observe that for $G$ as in (\ref{eqSolvab}) a $G$-torsor $Y\to X$ can be seen as a tower of commutative torsors, each of them being a $H_i/H_{i+1}$-torsor: they are called the commutative components of the solvable $G$-torsor. If for instance $n=2$  consider the contracted product $Y':=Y\times^G G/H_1$ (\cite{DemGab}, III, \S 4, n$^{\circ}$ 3) in order to factor $Y$ into a tower of two commutative torsors: a commutative $G/H_1$-torsor  $Y'\to X$ and a commutative $H_1$-torsor $Y\to Y'$:
$$\xymatrix{Y \ar[rr]^{H_1}\ar[dr]_{G} & & Y'\ar[dl]^{G/H_1}\\ & X & }$$ 
If $n>2$  we iterate the process factoring $Y\to Y'$ and so on. 

\subsection{Towers of  torsors}
\label{sez:ToT}
Now, let $S$ be a connected Dedekind scheme and $\eta=Spec(K)$ its generic point, let $X$ be a scheme, $f:X\to S$ a faithfully flat morphism of finite type and $f_{\eta}:X_{\eta}\to \eta$ its generic fiber. We assume the existence of a section $x\in X(S)$ and we denote by $x_{\eta}\in X_{\eta}(K)$ its generic fiber. We first consider the following general situation: we are given a finite $K$-group scheme $G$ (here $G$ is not necessarily solvable) and a $G$-torsor $Y\to X_{\eta}$ pointed in $y\in Y_{x_{\eta}}(K)$; we are looking for a  model of $G$, i.e. a finite and flat  $S$-group scheme $G'$ whose generic fiber is isomorphic to $G$ and a model of $Y\to X_{\eta}$, i.e. a $G'$-torsor $Z\to X$ whose generic fiber is isomorphic to $Y\to X_{\eta}$ as a $G$-torsor. Let $G_2$ be a non trivial $K$-finite  (but not necessarily commutative) closed, normal   subgroup scheme of $G$ and $G_1:=G/G_2$; we can see $Y\to X_{\eta}$ as a tower of two  torsors: a $G_2$-torsor $Y_2=Y\to Y_1$  and a  $G_1$-torsor  $Y_1\to X_{\eta}$ pointed in $y_1\in {Y_1}_x(K)$, image of $y$. 
We assume we are able to extend each  component, i.e. there exist finite and flat  $S$-group schemes $G_2'$ and $G_1'$ models (resp.) of $G_2$ and $G_1$, a $G_1'$-torsor $Z_1\to X$ extending $Y_1\to X_{\eta}$ and a  $G'_2$-torsor $Z_2\to Z_1$ extending $Y_2\to Y_1$  (pointed resp. in $z_1\in {Z_1}_x(S)$ and $z_2\in {Z_2}_{z_1}(S)$, sections extending $y_1$ and $y$). Then we are in the situation described by the following diagram:


\begin{equation}\label{diagTorsoriBrutti}
\xymatrix{Y=Y_2 \ar[d]_{G_2}\ar[r]& Z_2\ar[d]^{G_2'}\\ Y_1 \ar[d]_{G_1} \ar[r] & Z_1\ar[d]^{G_1'}\\ X_{\eta} \ar[d] \ar[r] & X \ar[d]\\Spec(K) \ar[r] & S}\end{equation}

 In general $Z_2\to X$ need not be a torsor, but from the tower $Z_2\to Z_1\to X$ we can obtain a torsor whose generic fiber is isomorphic to the $G$-torsor $Y\to X_{\eta}$, this is the object of the following:

\begin{teo}\label{teoEstTorri} The $G$-torsor $Y\to X_{\eta}$ can be extended to a finite pointed $G'$-torsor $Z\to X$ for some model $G'$ of $G$ if and only if the $G_1$-torsor $Y_1\to X_{\eta}$ and the $G_2$-torsor $Y\to Y_1$ can be extended. 
\end{teo}
\proof  The ``only if'' part is easy and left to the reader. So  consider the tower of torsors $Z_2\to Z_1\to X$, that exists by assumption. Then by a result of Garuti (\cite{Gar}, \S 2, Theorem 1) there exist flat $S$-group schemes of finite type $N$, $M$ and $H$, an $S$-scheme $T$ (provided with $t\in T_x(S))$ and morphisms $T\to Z_2$, $T\to Z_1$ and $T\to X$  which are respectively a $N$-torsor, $M$-torsor and $H$-torsor (all pointed), such that the following diagram commutes:

$$
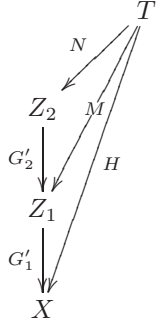
\xymatrix{ & T \ar[dl]_{N}\ar[ddl]|-{M}\ar[dddl]^{H} \\ Z_2\ar[d]_{G_2'}& \\Z_1\ar[d]_{G_1'} & \\X & }$$

\noindent then in particular there are canonical faithfully flat group scheme morphisms  $\gamma_2: M\to G_2'$ and $\gamma_1: H\to G_1'$ over $S$ where $M\simeq ker(\gamma_1)$ and $N\simeq ker(\gamma_2)$. First we observe that  $N$ is normal in $H$: indeed generically $N_{\eta} \unlhd H_{\eta}$ because $N_{\eta}$ is the kernel of the natural morphism $H_{\eta}\to G$; but $N$ coincides with the schematic closure of $N_{\eta}$ in $H$ then $N\unlhd  H$. Hence we can construct the quotient $H/N$, which is a $S$-flat group scheme (\cite{ANA} Th\'eor\`eme 4.C)
that fits in the following exact sequence (\cite{DemGab}, III, \S 3, n$^{\circ}$ 3, 3.7 a)) 
\begin{equation}\label{sequenzaFpqc}
\xymatrix{0\ar[r] & G_2'\ar[r] & H/N \ar[r] & G_1'\ar[r] & 0}
\end{equation}
then it is finite since $G_2'$ and $G_1'$ are (\cite{BER}, Proposition 9.2, (viii)). Let $\gamma:H\to (H/N)$ be the canonical faithfully flat morphism. Thus we  construct the contracted product $Z:=T\times^H (H/N)$ via $\gamma$ which is a $H/N$-torsor. The contracted product commuting with base change (\cite{DemGab}, III, \S 4, n$^{\circ}$ 3, 3.1), we have $Z_{\eta}:=(T\times^H (H/N))_{\eta}\simeq T_{\eta}\times^{H_{\eta}} H_{\eta}/N_{\eta}$ then in particular $T_{\eta}\times^{H_{\eta}} H_{\eta}/N_{\eta}\simeq Y$ as a $G$-torsor over $X$ hence $Z$ is a $H/N$-torsor over $X$ extending the starting one. 
\endproof

\begin{rem} Let $T_1$ and $T_2$ be (resp.) a $G_1$-torsor over $X$ pointed in $t_1\in {T_1}_x(S)$ and a $G_2$-torsor over $X$ pointed in $t_2\in {T_2}_x(S)$. Recall that a $X$-morphism $T_1\to T_2$ sending $y_1\mapsto y_2$ commutes necessarily with the actions of their structural group schemes. We have implicitly used this fact in previous lemma without mentioning it.
\end{rem}

\begin{rem}\label{remFinale} Keeping notations of theorem \ref{teoEstTorri} observe that $Z$ factors through $Z_1$ and in particular $Z\to Z_1$ is a  $G_2'$-torsor. Indeed 
$$ Z\times^{H/N}G_1'\simeq (T\times^H H/N)\times^{H/N}G_1' \simeq T\times^H G_1'\simeq Z_1 $$ then $Z\to Z_1$ is a $ker(H/N\to G_1')$-torsor.
\end{rem}

\begin{cor}\label{corEstensione2Torsori}  Let $G$ be a finite and solvable $K$-group scheme and $Y\to X_{\eta}$ a $G$-torsor.  Then  $Y\to X_{\eta}$ can be extended to a finite solvable $G'$-torsor $Z\to X$ for some model $G'$ of $G$ if and only if its commutative components  can be extended. 
\end{cor}
\proof The case where $G$ has a normal series of length $n=2$ is exactly theorem \ref{teoEstTorri}. With a little effort this procedure can be generalized to the case where $G$ has no normal series of length $n=2$, simply repeating Garuti's construction and theorem \ref{teoEstTorri}.
\endproof

\section{Extension of solvable torsors}
In the situation of diagram (\ref{diagTorsoriBrutti}) we now assume that $G_1$ and $G_2$ are commutative. In \cite{Antei2}, Theorem 3.1 we have explained how to extend finite quotient\footnote{Over any base scheme $S$ a pointed $G$-torsor $Y\to X$ over $S$ is said to be quotient if $X$ has a fundamental group scheme $\pi_1(X,x)$ (cf. for instance \cite{Antei3}, where the existence of the fundamental group scheme is studied) and the canonical morphism of $S$-group schemes $\pi_1(X,x)\to G$ is faithfully flat} pointed commutative torsors from $X_{\eta}$ to $X$ where $X$ needs to satisfy some strong assumptions (\cite{Antei2}, Notation 2.20). Thus for such $X$, it is not difficult  to find a finite, flat and commutative $S$-group scheme $G_1'$ as well as a $G_1'$-torsor   $Z_1\to X$ that extends the $G_1$-torsor $Y_1\to X_{\eta}$. Unfortunately, even if we can easily find schemes  $X$ satisfying these strong conditions (loc. cit. \S 3.2), it is improbable that  $Z_1\to S$ satisfy the same assumptions,  even in the case of curves: for instance it is asked $X$ to be smooth but in general $Z_1$ is not. So it is necessary to weaken the assumptions on $X$ hoping that $Z_1\to S$ is nice enough to be able to construct over $Z_1$ a torsor extending the $G_2$-torsor $Y\to Y_2$. 


\subsection{Commutative torsors}
%

For the sake of completeness we recall in a few lines the definition of Néron model and some properties which will be used in this paper. The reader can refer to \cite{BLR} for a deep discussion on the subject. Here we only consider Néron models of abelian varieties since it is the only case we will use. For the same reason the base scheme $S$ we consider will be the spectrum of a discrete valuation ring $R$:

\begin{defin}\label{definNeron} Let $R$ be a d.v.r. with fraction field $K$. Let $A$ be an abelian variety over $K$. A Néron model of $A$ is a smooth and separated $R$-scheme of finite type $\mathcal{N}_A$ whose generic fiber is isomorphic to $A$ and which satisfies the following universal property (called the Néron mapping property): for each smooth $R$-scheme $Y$ and each $K$-morphism $u:Y_{\eta}\to A$ there exists a unique morphism $u':Y\to \mathcal{N}_A$ extending $u$ where as usual $Y_{\eta}$ denotes the generic fibre of $Y$.
\end{defin}

\begin{prop}\label{propNeron} We keep notation of definition \ref{definNeron}. Then $A$ admits a Néron model $\mathcal{N}_A$ over $R$.
\end{prop}
\begin{proof}See for instance \cite{BLR}, \S 1.3, Corollary 2.
\end{proof}

By the Néron mapping property the Néron model $\mathcal{N}_A$ of $A$ is unique up to canonical isomorphism and it is a commutative group scheme.
Unfortunately in general $\mathcal{N}_A$ is not  an abelian scheme and not even a semi-abelian scheme. 
%
%
%
When  $\mathcal{N}_{A}$ is an abelian scheme then we simply say that $A$ has abelian (or good) reduction. If $\mathcal{N}_{A}$ is not an abelian scheme but there exists a finite Galois extension $L/K$ such that the Néron model $\mathcal{N}_{A_L}$ of $A_L:=A\times_{Spec(K)} {Spec(L)}$ is an abelian scheme over the integral closure $R'$ of $R$ in $K$ then we say that $A$ has potentially abelian (or potentially good) reduction. \\

Let $X$ be an $S$-scheme and $X\to S$ a proper morphism of finite type, then in what follows we denote by $Pic_{(X/S)(fppf)}$ the sheaf, in the fppf topology, associated to the relative Picard functor given by
$$Pic_{X/S}(T):=Pic(X\times_S T)/Pic(T)$$ for any $S$-scheme $T$ (see \cite{Antei2}, \S 2 for a brief introduction and \cite{KL} for a complete reference on this topic)\footnote{N.B.: here we have used Kleiman's notation. In \cite{BLR}, \S 8.1, Definition 2, however, our  $Pic_{(X/S)(fppf)}$ is called ``the relative Picard functor'' and denoted $Pic_{X/S}$.}. It is known that  for any $s\in S$ the sheaf $Pic_{(X_s/k(s))(fppf)}$ is represented by a group scheme $\mathbf{Pic}_{X_s/k(s)}$ whose identity component is denoted by $\mathbf{Pic}^0_{X_s/k(s) }$; over $S$ we denote by $Pic^0_{X/S}$ the subfunctor of $Pic_{(X/S)(fppf)}$ which consists to all elements whose restrictions to all fibers $X_s, s\in S$ belong to $\mathbf{Pic}^0_{X_s/k(s)}$. We recall the following result concerning the representability of $Pic^0_{X/S}$:

\begin{teo}\label{teoRaynaudBLR} Let $S$ be the spectrum of a d.v.r. $R$ and let $\eta:=Spec(K)$ be its generic point. Let $f:X\to S$ be a regular fibered surface (i.e. a projective flat morphism with $X$ an integral, regular  scheme of dimension $2$) with geometrically integral and smooth  generic fiber $X_{\eta}$ and provided with a section $x\in X(S)$. Then ${Pic}^0_{X/S}$ is represented by a  separated and smooth $S$-scheme $\mathbf{Pic}^0_{X/S}$ and coincides with the identity component of the Néron model of $J:=\mathbf{Pic}^0_{X_{\eta}/K}$.
\end{teo}

\begin{proof} First we recall that under these assumptions $J$ is an abelian variety. According to \cite{BLR}, \S 9.5 Remark 5 the existence of a section implies that the greatest common divisor of the geometric multiplicities of the irreducible components of the special fiber $X_s$ of $X$ in $X_s$ is one. Then by loc. cit. \S 9.5, Theorem 4, ${Pic}^0_{X/S}$ is represented by a  separated and smooth $S$-scheme $\mathbf{Pic}^0_{X/S}$ which coincides with the identity component of the Néron model of $J:=\mathbf{Pic}^0_{X_{\eta}/K}$. 
\end{proof}

Let us denote by $u:X_{\eta}\to J$  the canonical closed immersion (cf. for instance \cite{KL} Exercise 9.4.13) usually known as the Abel-Jacobi map. In next proposition we construct, when possible, a morphism $u':X\to \mathcal{N}_{J}$ whose generic fiber is isomorphic to $u$, where $\mathcal{N}_{J}$ denotes the   N\'eron model $\mathcal{N}_{J}$ of $J$.

\begin{prop}\label{PropAbelJacobi} Let $S$, $R$, $K$ be as in theorem \ref{teoRaynaudBLR}. Let $f:X\to S$ be a  regular fibered surface with geometrically integral and smooth  generic fiber $X_{\eta}$ and provided with a section $x\in X(S)$. Let $J$ be the Jacobian of $X_{\eta}$,  $\mathcal{N}_J$ its Néron model and $u:X_{\eta}\to J$  the canonical closed immersion. Assume moreover that $J$ has abelian reduction.
%
Then there exists a morphism $u':X\to \mathcal{N}_{J}$ whose generic fiber is isomorphic to $u$.
\end{prop}

\begin{proof} If $X$ were smooth this would be the Néron mapping property of the Néron model $\mathcal{N}_{J}$. Since in general this does not happen  then we argue as follows:  by assumption $\mathcal{N}_{J}$ is an  abelian scheme (thus proper), then  construct the schematic closure $C:=\overline{X_{\eta}}$ of $X_{\eta}$ in $\mathcal{N}_{J}$, i.e. the only closed subscheme of $\mathcal{N}_{J}$, flat over $S$ with generic fiber isomorphic to $X_{\eta}$. It is an integral scheme (\cite{EGAI} Proposition 9.5.9), proper over $S$ (because $\mathcal{N}_{J}$ is) whose special fiber is equidimensional of dimension one  (\cite{Liu}, Ch. 4, Proposition 4.16). Now we desingularize $C$, i.e. we construct a projective (\cite{Liu}, Ch. 8, Theorem 3.16) regular model $\widetilde{C}$ of $X_{\eta}$ and a morphism $\widetilde{C}\to C$ which is generically an isomorphism. In particular $\mathcal{N}_{J}\simeq \mathcal{N}_{J}^0\simeq \mathbf{Pic}^0_{X/S}\simeq \mathbf{Pic}^0_{\widetilde{C}/S}$ by theorem \ref{teoRaynaudBLR} and from $\widetilde{C}\to \mathbf{Pic}^0_{X/S}$ one obtains the desired morphism $u':X\to \mathcal{N}_{J}$. Indeed the morphism $\widetilde{C}\to \mathbf{Pic}^0_{X/S}$ is an element of $\mathbf{Pic}^0_{X/S}(\widetilde{C})$, then in particular this corresponds to an element 

$$\xi\in \frac{Pic(X\times \widetilde{C})}{Pic(\widetilde{C})}={Pic}_{X/S}(\widetilde{C})={Pic}_{(X/S)(fppf)}(\widetilde{C})$$ 
(use \cite{Liu},  Ch. 8, Corollary 3.6, (c) then apply  \cite{BLR}, \S 8.1 Proposition 4) but since $\widetilde{C}$ and $X$ are both regular then $Pic(X)\simeq Pic(X_{\eta})\simeq Pic(\widetilde{C})$ and consequently 

$$\frac{Pic(X\times \widetilde{C})}{Pic(\widetilde{C})}\simeq  \frac{Pic(X\times \widetilde{C})}{Pic(X)}=Pic_{\widetilde{C}/S}(X)={Pic}_{(\widetilde{C}/S)(fppf)}(X).$$ 
Hence, starting from $\xi$, we get a morphism $X\to Pic_{(\widetilde{C}/S)(fppf)}$ that on the generic and special fibers factors (resp.) through $u:X_{\eta}\to J$ and $X_s\to \mathbf{Pic}^0_{\widetilde{C}_s/k(s)}$ (here $s$ denotes the special point) since $X$ has geometrically connected fibers, thus obtaining a morphism $X\to \mathbf{Pic}^0_{\widetilde{C}/S}$ that composed with $\mathbf{Pic}^0_{\widetilde{C}/S}\simeq \mathcal{N}_{J}^0\simeq \mathcal{N}_{J}$, gives the desired morphism $u':X\to \mathcal{N}_{J}$ extending $u:X_{\eta}\to J$ as described by the following diagram:
$$\xymatrix{X_{\eta} \ar[dr]_{u}\ar[r]\ar[dd]& X \ar@{-}[d]  \ar[dr]^{u'} & \\ &  J \ar[d] \ar[dl] \ar[r] & \mathcal{N}_J \ar[dl] \\ Spec(K)  \ar[r] & S & }$$

\end{proof}

A result due to Raynaud, that we state in our setting in the following theorem, shows that the hypothesis of proposition \ref{PropAbelJacobi} are satisfied in many relevant cases after eventually extending scalars:

\begin{teo}\label{teoReno} Let $R$ be a complete d.v.r. with residue characteristic $p>0$ and fraction field $K$. Let $X$ be a  smooth fibered surface  over $R$ with geometrically connected generic fiber and provided with a section $x\in X(S)$. Let $G$ be a finite and \'etale $K$-group scheme of order $p^n$ and $Y\to X_{\eta}$ a quotient $G$-torsor over the generic fiber of $X$, then the Jacobian $J_{Y}$ of $Y$ has potential abelian reduction. In particular every commutative component $Y_i$ of $Y$ has a Jacobian $J_{Y_i}$ with potential abelian reduction.
\end{teo}
\proof It is known that $G$ becomes constant after a finite Galois extension $L$ of $K$. Moreover since $X_{\eta}$ is geometrically connected then so is $X_s$ (\cite{Liu}, Ch. 8, Corollary 3.6, (c)), then finally use  \cite{Ray3}, Th\'eor\`eme 1.
\endproof 

We conclude this section with a result that will be used later:

\begin{teo}\label{teoExtAbel} Let $S$ be the spectrum of a d.v.r. $R$ and let $\eta:=Spec(K)$ be its generic point. Let $f:X\to S$ be a regular fibered surface provided with a section $x\in X(S)$.  Assume that  $f$ has  smooth generic fiber $X_{\eta}\to \eta$. Assume moreover that the Jacobian $J$ of $X_{\eta}$ has abelian reduction.  Then every finite, quotient, commutative, pointed  torsor over $X_{\eta}$ can be extended to a finite commutative pointed torsor over $X$.
\end{teo}

\begin{proof} 
Let $\mathcal{N}_J$ be the Néron model of $J$ and $u':X\to \mathcal{N}_J$ the morphism obtained in proposition \ref{PropAbelJacobi}. 
Let $G$ be a finite and flat $K$-group scheme, then according to \cite{Antei2}, Corollary 3.8 we know that every finite, quotient and commutative  $G$-torsor $T'\to X_{\eta}$  (pointed over $x_{\eta}\in X_{\eta}(K)$, generic fiber of $x$) is the pull back of a finite, quotient and commutative  $G$-torsor $T\to J$ (pointed over $0_{J}$). Now it is easy to find an $R$-model $H$ of $G$ (commutative, finite and flat) and a pointed (over $0_{\mathcal{N}_J}$) $H$-torsor $Y\to \mathcal{N}_J$ whose generic fiber is isomorphic to $T\to J$ (cf. for instance \cite{Antei}, \S 2.2). Then finally $Y':=Y\times_{\mathcal{N}_J} X$, the pull back over $u'$, is a finite, commutative  $H$-torsor over $X$  (pointed over $x$) extending $T'\to X_{\eta}$. 
\end{proof}

\subsection{Solvable torsors over curves}
\label{sez:solvable}
\begin{nota}\label{notaSolva}
From now on $S$ will be the spectrum of a complete discrete valuation ring $R$ with algebraically closed residue field $k$ of positive characteristic $p$ and with fraction field $K$. We will denote by  $\eta:=Spec(K)$ and $s:=Spec(k)$  respectively the generic and special points. Moreover $f:X\to S$ will be a regular fibered surface provided with a section $x\in X(S)$ with smooth and geometrically connected  (then geometrically integral) generic fiber $X_{\eta}\to Spec(K)$, pointed in $x_{\eta}$. Using a standard convention we say that a $S$-morphism of schemes $Z_1\to Z_2$ is a model map if it is generically an isomorphism.
\end{nota}

\begin{rem}Let $Y$ be any fibered surface over $S$. If we have a section $y\in Y(S)$ its generic fiber $Y_{\eta}$ is geometrically connected if and only if it is connected; moreover if $Y_{\eta}$ is geometrically reduced (resp. geometrically irreducible) then $Y_{\eta}$ is reduced (resp. irreducible) (\cite{Liu}, Ch. 3, \S 2, ex. 2.11 and 2.13).  Of course the same is true for the special fiber $Y_s$. Finally we recall that if $Y_{\eta}$ is integral then so is $Y$ (\cite{EGAI} Proposition 9.5.9). The N\'eron blowing up of $Y$ at a closed subscheme $C$ of $Y_s$ will be denoted by $Y^C$: the reader should refer to \cite{BLR}, \S 3.2, \cite{WW2}, \S 1 or \cite{ANA}, \S 2.1 for the definition and properties. We are not making any assumption on the characteristic of $K$.
\end{rem}

Before stating the principal result we need some preliminary lemmas. Lemma \ref{lemmaStandard}, as recalled in its proof, slightly generalizes \cite{WW2} Theorem\footnote{This result is stated by Waterhouse and Weisfeiler only for affine group schemes but, as observed by the authors, the group structure is never used (\cite{WW2}, page 552, Remark (4)).}  1.4, that we strongly use. 

\begin{lem}\label{lemmaStandard} Let $Y$ and $\widetilde{Y}$ be two schemes faithfully flat and of finite type  over $S$ and  $h:\widetilde{Y}\to Y$  an affine model map. Then $h$ is isomorphic to a composite of a finite number of N\'eron blowing ups.

\end{lem}
\begin{proof}First we observe that if the special fiber $h_s:\widetilde{Y}_s\to Y_s$ of $h$ is a schematically dominant morphism (i.e. $\mathcal{O}_{Y_s}\hookrightarrow {h_s}_{*}(\mathcal{O}_{\widetilde{Y}_s})$ is injective) then $h$ is an isomorphism: indeed let $U=Spec(A)$ be any open affine subset of $Y$ and  $V=Spec(A'):=h^{-1}(U)$ then consider $h_{|V}:V\to U$ and its special fiber $(h_{|V})_s:V_s\to U_s$ where $V_s=Spec(A'_k)=Spec(A'\otimes_R k)$ and $U_s=Spec(A_k)=Spec(A'\otimes_R k)$. We are thus reduced to consider the affine case, then one just needs to argue as in \cite{WW2}, Lemma 1.3. \\
Now we prove the statement of the lemma: if $h_s:\widetilde{Y}_s\to Y_s$ is schematically dominant there is nothing to do, otherwise consider  the scheme theoretic image $C_1:=h_s(\widetilde{Y}_s)$ of $\widetilde{Y}_s$ in $Y_s$. It  is a closed subscheme of $Y_s$ (\cite{EGAI}, \S 9.5). Now consider the N\'eron blowing up $Y^{C_1}$ of $Y$ in $C_1$  then $h$ factors through $Y_1:=Y^{C_1}$. Denote by $h_1:Y\to Y_1$ the  $S$-morphism obtained. If its special fiber ${(h_1)}_s$ is schematically dominant then $h_1:Y\simeq Y_1$ otherwise we set $C_2:={(h_1)}_s(\widetilde{Y}_s)$, $Y_2:=Y_1^{C_2}$ and we continue as before. Hence we conclude that $\widetilde{Y}\simeq \underleftarrow{lim}_iY_i$. If $Y$ and $\widetilde{Y}$ are affine then one argues as in \cite{WW2} Theorem 1.4 to conclude that we can stop after a finite number of steps, i.e. there exist $n\geq 0$ such that $\widetilde{Y}\simeq Y_n$.
If $Y$ and $\widetilde{Y}$ are not affine then let $\left\{U_j\right\}_{j \in J}$ be an affine open cover of $Y$ and $\left\{V_j:=h^{-1}(U_j)\right\}_{j \in J}$ the induced affine open cover of $\widetilde{Y}$. Since $Y$ is quasi compact
we can take $|J|<\infty$. To give $h$ is equivalent to give the family of morphisms 
$$\xymatrix{\{h_j:=(V_j \ar[r]^(.6){h_{|V_j}} & U_j\ar[r] & Y)\}_{j \in J}}$$
where we have given the $V_j$ and $U_j$ the induced subscheme structure (\cite{Har}, II, Theorem 3.3, Step 3). For any $j\in J$ set $C^1_j:=C_1\times_Y U_j$ (the scheme theoretic image $h_{|V_j}(V_j)$), $U_j^1:={U_j}^{C^1_j}$ and so on: it follows that $V_j\simeq \underleftarrow{lim}_i {U_j}^{i}$ 
but since $U_j$ and $V_j$ are affine then the projective limits become stable after  $n(j)\geq 0$ steps. Take $n:=max_{j\in J}\{n(j)\}$: this is the number of steps after which we can stop.\end{proof}

\begin{lem}\label{lemmaNerBlUpPrelim} Let $Y$ be a scheme faithfully flat and of finite type over $S$, $C_2$  a closed subscheme of $Y_s$ and $C_1$ a closed subscheme of $C_2$. Denote by $Y^{C_i}$ the N\'eron blowing up of $Y$ in $C_i$ ($i=1,2$). Let $C':=Y^{C_2}\times_{C_2}C_1$ the induced closed subscheme of ${(Y^{C_2})}_s$ then $Y^{C_1}\simeq {(Y^{C_2})}^{C'}$.
\end{lem}
\begin{proof}This follows directly from the universal property of the N\'eron blowing up and the following diagram:
$$\xymatrix{ {(Y^{C_1})}_s \ar@{-->}[rr]\ar[dr]& & C' \ar[ld]\ar@{^{(}->}[rr] & & {(Y^{C_2})}_s \ar[ld]\\ & C_1 \ar@{^{(}->}[rr]\ar@{^{(}->}[rd]& & C_2\ar@{^{(}->}[ld] & \\ & & Y_s & &  }$$
\end{proof}

\begin{lem}\label{lemmaNerBlUp} Let $Y$ be an integral fibered surface. Let $f:Y\to X$ be a finite and flat morphism, $C$ a closed subscheme of $Y_s$ and $Y^C$ the N\'eron blowing up of $Y$ in $C$. Assume that the canonical morphism $h:Y^C\to Y$ is a finite model map. Then there exist a regular fibered surface $X'$ and a finite model map  $X'\to X$ such that  $Y^C\simeq Y\times_X X'$.
\end{lem}
\begin{proof} Let $f_s:Y_s\to X_s$ be the special fiber of $f$ and $D_1:=f_s(C)$ the scheme theoretic image of $C$: it is a closed subscheme of $X_s$. Now consider the fiber product $C_1:=D_1\times_{X_s}Y_s$ and the natural closed immersion $C\hookrightarrow C_1$: if it is an isomorphism then, by the universal propertiy of the N\'eron blowing up, $Y^C\simeq X^{D_1}\times_X Y$ hence $X':=X^{D_1}$ is the required solution. Otherwise let $Y_1:=Y^{C_1}$, $X_1:=X^{D_1}$ and $f_1:Y_1\to X_1$ the pull back of $f$ over $X_1\to X$. The morphism $Y^C\to Y$ now factors through $Y_1$; then we analyze the morphism $Y^C\to Y_1$: by lemma \ref{lemmaNerBlUpPrelim} $Y^C\simeq Y_1^{C_1'}$ where $C_1':=C\times_{C_1}{(Y_1)}_s$, thus we are in the same situation as before: let $D_2:={(f_1)}_s(C_1')$, $C_2:=D_2\times_{{(X_1)}_s}{(Y_1)}_s$, $Y_2:={Y_1}^{C_2}$, $X_2:=X_1^{D_2}$, $C_2':=C_1'\times_{C_2}{(Y_2)}_s$ and so on. We finally obtain the isomorphism $Y^C\simeq \underleftarrow{lim}_i Y_i$ (where $Y_0:=Y$).
Now using arguments similar to those used in the last part of the proof of lemma \ref{lemmaStandard}  we are reduced to study the case where $X$ (then also $Y$ and $Y^C$) is affine: so let us set $Y_i:=Spec(A_i)$ and $Y^C=Spec(B)$ then since every $A_i$ is integral the  morphisms $Y_i\to Y_{i-1}$ induce a sequence of inclusions $$A_0\subseteq A_1\subseteq A_2 \subseteq .. \subseteq A_i \subseteq .. \subseteq B;$$
since $B$, as a $A_0$-module, is finite then it is generated by a finite number of elements $\{b_j\}_j\subset B$ so there exists an integer $n\geq 0$ such that $\{b_j\}_j\subset A_n$. Hence  $Y^C\simeq Y_n$ and $X':=X_n$ allows us to conclude.
\end{proof}

\begin{rem}In lemma \ref{lemmaNerBlUp} we never use the assumption that the absolute dimension of $X$ is $2$, but it is the only case of interest in this paper. The assumption that the residue field is algebraically closed will be used in lemma \ref{lemmaScoppio} and not before.
\end{rem}

\begin{lem}\label{lemmaFiniteModelMap} Let $f:Y\to X$ be a finite and flat morphism with $Y$ integral. Let $h:\widetilde{Y}\to Y$ be a finite model map. Then there exist a regular fibered surface $\widetilde{X}$ and a model map $\widetilde{X}\to X$ such that  $\widetilde{Y}\simeq Y\times_X \widetilde{X}$. Moreover $\widetilde{X}\to X$ is isomorphic to a composite of a finite number of N\'eron blowing ups.
\end{lem}
\begin{proof}This is just a consequence of lemmas \ref{lemmaStandard} and  \ref{lemmaNerBlUp}.
\end{proof}

\begin{cor}\label{corNormalizzo} Let $f:Y\to X$ be a finite and flat morphism with $Y$ integral. Let $h:\widetilde{Y}\to Y$ be the normalization morphism. Assume that the generic fiber $Y_{\eta}$ of $Y$ is smooth and geometrically integral. Then there exist a regular fibered surface $\widetilde{X}$ and a model map $\widetilde{X}\to X$ such that  $\widetilde{Y}\simeq Y\times_X \widetilde{X}$.
\end{cor}
\begin{proof}In this context the normalization morphism $h:\widetilde{Y}\to Y$  is a finite (then affine) model map (\cite{Liu}, Ch 8, Lemma 3.49). Then the result is just a consequence of lemma \ref{lemmaFiniteModelMap}.
\end{proof}

\begin{lem}\label{lemmaScoppio} Let $G$ be a finite and flat $S$-group scheme with infinitesimal special fiber $G_s$ and $f:Y\to X$ a $G$-torsor. Assume that the generic fiber $Y_{\eta}$ of $Y$ is smooth and geometrically integral. Let moreover  $\widetilde{Y}\to Y$ be the blowing-up of $Y$ centered at a point $q$ of the special fiber $Y_s$ of $Y$. Then $\widetilde{Y}\simeq Y\times_X \widetilde{X}$ where $\widetilde{X}\to X$ is the blowing up centered at $p:=f_s(q)$.
\end{lem}
\begin{proof} The residue field $k$ being algebraically closed then $q:Spec(k)\to  Y$ and also $p:Spec(k)\to X$ (\cite{Liu}, Ch.2 ex. 5.9). Thus, since $f_s:Y_s\to X_s$ is a $G_s$-torsor, $p\times_XY\simeq G_s$ and the canonical closed immersion $q\to G_s$ identifies $q$ with $(G_s)_{\text{red}}$ (recall that $G_s$ is infinitesimal). Then the blowing up $Y'$ of $Y$ centered at $p\times_XY$ is isomorphic to the blowing up of $Y$ centered at $q$ (\cite{Liu}, Ch. 2, ex. 3.11 (a)). But since $Y'\simeq Y\times_X \widetilde{X}$ (\cite{Liu}, Ch. 8, Proposition 1.12 (c)) then  $\widetilde{Y}\simeq Y\times_X \widetilde{X}$, as required.
\end{proof}

\begin{rem}\label{remDesing}
Let $Y$ be any fibered surface over $S$ with smooth generic fiber $Y_{\eta}$: the canonical desingularization of  $Y$ is the sequence of blowing ups
\begin{equation}\label{eqLipman} .. \to Y_i\to Y_{i-1}\to Y_{i-2} \to .. \to Y_1\to Y_0=Y \end{equation}
where for each $i$, the morphism $Y_i\to Y_{i-1}$ denotes 
\begin{itemize}
	\item the normalization morphism if $i$ is odd (it can eventually be an isomorphism if $Y_{i-1}$ is already normal);
	\item the blowing up at the singular points of $Y_{i-1}$ if $i$ is even. 
\end{itemize}
Recall that at each step, when $i$ is even, the set $Sing(Y_{i-1})$ of singular points of $Y_{i-1}$ is a finite set of points contained in the special fiber ${(Y_{i-1})}_s$. According to \cite{Liu}, Ch.8, Corollary 3.51, there exists an integer $n\geq 0$ such that $\widetilde{Y}:=Y_n$ is regular and the morphism $\widetilde{Y}\to Y$ is a model map.
\end{rem}

\begin{prop}\label{propScoppio} Let $G$ be a finite and flat $S$-group scheme with infinitesimal special fiber $G_s$ and $f:Y\to X$ a $G$-torsor. Assume that the generic fiber $Y_{\eta}$ of $Y$ is smooth and geometrically integral. Let moreover  $\widetilde{Y}\to Y$ be the canonical desingularization of  $Y$. Then there exist a regular fibered surface  $\widetilde{X}$ and a morphism $\widetilde{X}\to X$ such that $\widetilde{Y}\simeq Y\times_X \widetilde{X}$. In particular $\widetilde{Y}\to \widetilde{X}$ is a $G$-torsor.
\end{prop}
\begin{proof} According to previous discussion $\widetilde{Y}\to Y$ is a sequence of normalization morphisms (which are finite morphisms) and blowing ups centered at a finite set of singular points. Then in order to conclude it is sufficient to use lemma \ref{lemmaScoppio} and corollary \ref{corNormalizzo}.
\end{proof}

%

Before stating the main theorem of this paper we need a last lemma:

\begin{lem}\label{lemmaPRINCIPALE} Let $Z\to X$ be a finite $(\mathbb{Z}/p\mathbb{Z})_R$-torsor. Then there exist a finite and flat $R$-group scheme $G$ with infinitesimal special fiber, a $G$-torsor $Y\to X$ and a model map $\varphi:Z\to Y$ commuting with the actions of $(\mathbb{Z}/p\mathbb{Z})_R$ and $G$.
\end{lem}
\begin{proof} That a model map $\rho:(\mathbb{Z}/p\mathbb{Z})_R\to G$ such that $G_{s}$ is infinitesimal exists is clear from \cite{Mau}, \S 3.2 when $char(K)=p$ and from \cite{OOS}, I, \S 2, when $char(K)=0$, then the model map $\varphi:Z\to Y$ is given by the contracted product (through $\rho$) $Y=Z\times^{(\mathbb{Z}/p\mathbb{Z})_R} G$ .
\end{proof}

\begin{rem}The $G$-torsor $Y\to X$ obtained in lemma \ref{lemmaPRINCIPALE} has trivial special fiber but this will not affect the following discussion.
\end{rem}

\begin{teo}\label{teoPRINCIPALE} Let $X$ be a proper and smooth fibered surface over $R$ with geometrically connected fibers and provided with a section $x\in X(S)$.  Let $G$ be a finite, étale, solvable $K$-group scheme of order $p^n$ and $Y\to X_{\eta}$ a quotient $G$-torsor, pointed in $y\in Y_x(K)$. Then, after eventually a finite extension of scalars, there exist a regular fibered surface $\widetilde{X}$, a model map $\widetilde{X}\to X$, a finite flat and solvable $R$-group scheme $G'$ of order $p^n$ such that  $Y\to X_{\eta}$  can be extended to a $G'$-torsor  $Y'\to \widetilde{X}$. Moreover we can construct $Y'$ in such a way to make it regular.
\end{teo}

\begin{proof} First of all we observe that we can decompose  $Y\to X_{\eta}$ into a tower of $n$ torsors $Y_1\to X_{\eta}$, $Y_i\to Y_{i-1}$ (for $i=2,.., n$, where $Y_n=Y$) each one being a quotient pointed $G_i$-torsor where $|G_i|=p$.  After eventually extending scalars (as explained in theorem \ref{teoReno}), we can assume that $G_i\simeq (\mathbb{Z}/p\mathbb{Z})_K$ (for all $i=1,..,n$) and that the Jacobian $J_{Y_i}$ has abelian reduction. Assume first that $n=2$: according to theorem \ref{teoExtAbel} there exist a finite and flat $R$-group scheme $G'_1$ of order $p$, generically isomorphic to  $G_1$, and a $G_1'$-torsor $Z_1\to X$ extending $Y_1\to X_{\eta}$. We can assume by lemma \ref{lemmaPRINCIPALE} that ${(G_1)}_s$ is infinitesimal. If $Z_1$ is regular we go on extending $Y_2\to Y_1$, otherwise we  desingularize $Z_1$ as recalled in remark \ref{remDesing}, i.e. we find a regular fibered surface $Z_1'$ and a model map $Z_1'\to Z_1$. Moreover by proposition \ref{propScoppio} there exist a regular fibered surface $X'$ and a model map $X'\to X$ such that $Z_1'\to X'$ is a $G_1'$-torsor. Now we proceed as before: there exist a finite and flat $R$-group scheme $G'_2$ of order $p$, generically isomorphic to  $G_2$, and a $G_2'$-torsor $Z_2\to Z_1'$ extending $Y_2\to Y_1$. Again we can  assume that ${(G_2)}_s$ is infinitesimal. Then by theorem \ref{teoEstTorri} there exist a finite, flat, infinitesimal $S$-group scheme $G'$ generically isomorphic to $G$ and a $G'$-torsor $Z\to X'$ extending $Y\to X_{\eta}$ and we are done setting $Y':=Z$. We only mention how to proceed when $n>2$: we start from $Z$ and, as before, we desingularize it, i.e.  we find a regular fibered surface $Z_2'$ and a model map $Z_2'\to Z$. As before there exist a regular fibered surface  $X''\to X'$ such that $Z_2'\to X''$ is a $G_1'$-torsor; then we can extend $Y_3\to Y_2$ to a torsor over $Z_2'$ and so on. We argue in the same way to prove that we can find a regular $Y'$ (if it is not we desingularize, etc.).
\end{proof}

\begin{cor}\label{corLunghezzaDue} Let $X$ be a proper and smooth  fibered surface over $R$ with geometrically connected fibers and provided with a section $x\in X(S)$. Let $G$ be a finite, étale, $K$-group scheme having a normal series of length $n=2$. Let $Y\to X_{\eta}$ be a quotient $G$-torsor, pointed in $y\in Y_x(K)$. Then, after eventually extending scalars, there exist a regular fibered surface $\widetilde{X}$, blowing up of $X$ at a closed subcheme of $X_s$, a finite flat and solvable $R$-group scheme $G'$ such that  $Y\to X_{\eta}$  can be extended to a $G'$-torsor over  $\widetilde{X}$. 
\end{cor}
\begin{proof} We can assume that the $K$-group scheme $G$ is constant (it is always true after eventually extending scalars). Let us decompose  $Y\to X_{\eta}$ into a tower of two commutative torsors: a $G_1$-torsor $Y_1\to X_{\eta}$ and a $G_2$-torsor $Y\to Y_1$. If $p \nmid |G_1|$ then the problem has an easy answer, otherwise let $p^n$ be the maximal $p$-power dividing $|G_1|$ and ${}^pG_1$ a (normal) $K$-subgroup of $G_1$ of order $p^n$. Then  the Jacobian $J_{Y_1}$ of $Y_1$ has potentially abelian reduction. Indeed  $Y_1$ can be decomposed into a tower of two torsors: a ${}^pG_1$-torsor $Y_1\to T$ and a  $G_1/{}^pG_1$-torsor $T\to X_{\eta}$. The latter can be extended, after eventually extending scalars, to a finite and \'etale torsor $T'\to X$ (we refer the reader to the introduction of this paper) then we apply theorem \ref{teoReno} to $Y_1\to T$. Now we forget this decomposition for $Y_1\to X_{\eta}$ and we assume that over $K$ the Jacobian $J_{Y_1}$ has abelian reduction (we have seen it is always true after eventually extending scalars). We would rather consider the following decomposition for $Y_1\to X_{\eta}$ as a tower of two torsors: a ${}^pG_1$-torsor $P\to X_{\eta}$ and a $G_1/{}^pG_1$-torsor $Y_1\to P$. Theorem \ref{teoReno} tells us that the Jacobian $J_P$ of $P$ has potentially abelian reduction; again we can assume that it has in fact abelian reduction. Hence according to theorem \ref{teoPRINCIPALE} there exist a regular fibered surface $\widetilde{X}$, a model map $\widetilde{X}\to X$, a finite flat and commutative $R$-group scheme $H_1$ such that  $P\to X_{\eta}$  can be extended to a $H_1$-torsor  $P'\to \widetilde{X}$ with $P'$ regular. Furthermore by theorem \ref{teoExtAbel} there exists a finite flat and commutative $R$-group scheme $H_2$ such that $Y_1\to P$ can be extended to a $H_2$-torsor $Y_1'\to P'$; by theorem \ref{teoEstTorri} there exist a finite and flat $S$-group scheme $H$ generically isomorphic to $G_1$ and a $H$-torsor $Z\to \widetilde{X}$ extending $Y_1\to X_{\eta}$. Since $p\nmid |H_2|$ then $H_2$ is \'etale; moreover $Z\to \widetilde{X}$ factors through $P'$, more precisely  $Z\to P'$ is  a $H_2$-torsor (remark \ref{remFinale}) so $Z\to P'$ is smooth, then $Z$ is regular as $P'$ is (see for instance \cite{BLR}, \S 2.3 Proposition 9). Finally we can apply again theorem \ref{teoExtAbel} to $Y\to Y_1$ and \ref{teoEstTorri} in order to conclude.
\end{proof}

\begin{rem}It is obvious that the tools we have presented allows us to extend solvable torsors even if they do not have a normal series of length $2$ but only in some particular cases, for example if every commutative component $Y_i$ of the  torsor $Y\to X_{\eta}$ has a Jacobian $J_{Y_i}$ that has potentially abelian reduction. As clear from the proof of corollary \ref{corLunghezzaDue} this condition is satisfied, for instance, when all the $G_i$ but $G_1$ have order not divisible by $p$.
\end{rem}

\medskip

\scriptsize

\begin{flushright} Marco Antei\\ 
E-mail: \texttt{antei@math.univ-lille1.fr}\\ \texttt{marco.antei@gmail.com}\\
\end{flushright}


\begin{thebibliography}{99}

\bibitem{ANA} {\sc S.~Anantharaman}, \emph{Sch\'emas en groupes, espaces homog\`enes et espaces alg\'ebriques sur une base
de dimension 1}. M\`emoires de la S. M. F., tome 33, (1973) 5-79.


\bibitem{Antei} {\sc M.~Antei}, \emph{Comparison between the Fundamental Group Scheme of a Relative Scheme $X$ and that of
its Generic Fiber}, Journal de théorie des nombres de Bordeaux, Tome 22, no 3 (2010), p. 525-543.

\bibitem{Antei2} {\sc M.~Antei}, \emph{On the Abelian Fundamental Group Scheme of a Family of Varieties}, Israel Journal of Mathematics, (to appear) (2010).

\bibitem{Antei3} {\sc M.~Antei}, \emph{The  Fundamental Group Scheme of a non Reduced Scheme}, arXiv:1011.5596v1, (2010).





\bibitem{BER} {\sc J. E.~Bertin}, \emph{G\'en\'eralites sur les
pr\'esch\'emas en groupes}. \'Expos\'e VI$_B$, \emph{S\'eminaires de
g\'eom\'etrie alg\'ebrique du Bois Marie}.  III , (1962/64)


\bibitem{BLR} {\sc S.~Bosch, W.~L\"utkebohmert, M.~Raynaud}, \emph{N\'eron models}, Springer Verlag, (1980).






\bibitem{DemGab} {\sc M.~Demazure, P.~Gabriel}, \emph{Groupes alg\'ebriques}, North-Holland Publ. Co., Amsterdam, (1970).


%


\bibitem{Gar} {\sc M. A.~Garuti}, \emph{On the ``Galois closure'' for Torsors}, Proc. Amer. Math. Soc. 137, 3575-3583 (2009).


\bibitem{Gas} {\sc C.~Gasbarri}, \emph{Heights of Vector Bundles and the Fundamental Group Scheme of a Curve}, Duke Mathematical Journal, Vol. 117, No. 2, (2003)  287-311.

\bibitem{EGAI} {\sc A.~Grothendieck}, \emph{\'{E}l\'ements de G\'eom\'etrie
Alg\'ebrique}. I. \emph{Le langage des schémas}. Publications Math\`ematiques de l'IHES, 4,
(1960).


\bibitem{EGAII} {\sc A.~Grothendieck}, \emph{\'{E}l\'ements de G\'eom\'etrie
Alg\'ebrique}. II. \emph{Étude globale \`el\'ementaire de quelques
classes de morphisms}. Publications Math\`ematiques de l'IHES, 8,
(1961).


\bibitem{EGAIV-2} {\sc A.~Grothendieck}, \emph{\'{E}l\'ements de g\'eom\'etrie
alg\'ebrique}. IV. \emph{\'{E}tude locale des sch\'emas et des morphismes
de sch\`emas}. II, Publications Math\'ematiques de l'IHES, 24, (1965).

\bibitem{SGA1} {\sc A.~Grothendieck}, \emph{Revêtements \'etales et groupe fondamental}, S\'eminaire de g\'eom\'etrie alg\'ebrique du Bois Marie, (1960-61).
%
%
%

\bibitem{Har} {\sc R.~Hartshorne}, \emph{Algebraic Geometry}, Graduate
Texts in Mathematics, Springer, (1977).

\bibitem{KL} {\sc S. L.~Kleiman}, \emph{The Picard Scheme}, Fundamental Algebraic Geometry, AMS, (2005).


\bibitem{Liu} {\sc Q.~Liu}, \emph{Algebraic Geometry and Arithmetic
Curves}, Oxford Science Publications (2002)




%

\bibitem{Mau} {\sc S.~Maugeais}, \emph{Rel\`evement des rev\^etements $p$-cycliques des courbes rationnelles semistables},
Math. Ann. 327, No.2, 365-393 (2003).


%



\bibitem{mumGIT}  {\sc D.~Mumford, J.~Fogarty},  \emph{Geometric Invariant Theory}, Springer-Verlag, (1982).







\bibitem{OOS} {\sc F. Oort, T. Sekiguchi, N. Suwa}, \emph{On the deformation of Artin-Schreier to Kummer},
Ann. Sci. \'Ec. Norm. Sup. (4-\`eme s\'erie) 22, No.3, 345-375 (1989).


%

\bibitem{Ray3} {\sc M.~Raynaud}, \emph{$p$-groupes et réduction semi-stable des courbes}, The Grothendieck Festschrift, Vol III, Progr. Math., vol. 88, Birkh\"auser, Boston, MA, (1990), p. 179-197.

\bibitem{Saidi} {\sc M. Sa\"idi}, \emph{Torsors under finite and flat group schemes of rank $p$ with Galois action}, Math. Zeit. 245, no. 4 (2003), p. 695-710. 


\bibitem{Sha} {\sc S. S.~Shatz}, \emph{Group Schemes, Formal Groups, and $p$-Divisible Groups}, on \emph{Arithmetic Geometry}, Springer-Verlag, (1995), p. 29-78. 

%
\bibitem{Szamuely} {\sc T.~Szamuely}, \emph{Galois Groups and Fundamental Groups}, Cambridge Studies in Advanced Mathematics, vol. 117,  Cambridge University Press (2009)

%
\bibitem{Tos} {\sc D.~Tossici} \emph{Effective Models and Extension of Torsors over a d.v.r. of Unequal Characteristic}, International Mathematics Research Notices (2008) Vol. 2008 : article ID rnn111, 68 pages (2008).

\bibitem{WW2} {\sc W.C.~Waterhouse, B. Weisfeiler}, \emph{One-Dimensional Affine Group Schemes}, Journal of Algebra, 66, 550-568 (1980).

%



%



\end{thebibliography}
\end{document}